\documentclass{article}
\usepackage{amsmath, amsfonts, amsthm, amssymb}
\usepackage{graphicx}
\usepackage{float}
\usepackage{verbatim}
\usepackage{color}
\usepackage[active]{srcltx}  

\numberwithin{equation}{section}

\newtheorem{thm}{Theorem}[section]
\newtheorem{lma}[thm]{Lemma}

\newtheorem{defn}[thm]{Definition}

\newtheorem{prop}[thm]{Proposition}

\renewcommand{\ge}{\geqslant}
\renewcommand{\le}{\leqslant}
\renewcommand{\geq}{\geqslant}
\renewcommand{\leq}{\leqslant}
\renewcommand{\H}{\text{H}}
\renewcommand{\P}{\text{P}}

\newcommand{\ol}{\overline}
\newcommand{\e}{\varepsilon}
\newcommand{\ut}{\underline{t}}

\allowdisplaybreaks

\title{On the dimensions of a family of overlapping self-affine carpets}

\author{Jonathan M. Fraser$^1$ and Pablo Shmerkin$^2$}

\begin{document}
\maketitle

\begin{center}

\emph{Dedicated to the memory of Dave Broomhead.}
\\ \vspace{7mm}
  $^1$School of Mathematics, The University of Manchester, Manchester, M13 9PL, UK.\\  E-mail: jonathan.fraser@manchester.ac.uk \\ \vspace{3mm} $^2$Department of Mathematics and Statistics, Torcuato Di Tella University, Av. Figueroa Alcorta 7350, Buenos Aires, Argentina. E-mail: pshmerkin@utdt.edu
\end{center}

\begin{abstract}
We consider the dimensions of a family of self-affine sets related to the Bedford-McMullen carpets.  In particular, we fix a Bedford-McMullen system and then randomise the translation vectors with the stipulation that the column structure is preserved.  As such, we maintain one of the key features in the Bedford-McMullen set up in that alignment causes the dimensions to drop from the affinity dimension.  We compute the Hausdorff, packing and box dimensions outside of a small set of exceptional translations, and also for some explicit translations even in the presence of overlapping. Our results rely on, and can be seen as a partial extension of, M. Hochman's recent work on the dimensions of self-similar sets and measures.\\


\end{abstract}

\section{Introduction}

\subsection{Self-affine sets and carpets}

The dimension theory of self-affine sets has attracted a great deal of attention in the literature over the past 30 years.  There are two key starting points which have led to two thriving and complementary strands of research.  The `generic case' studies general self-affine sets by randomising the translation vectors in the defining iterated function system in an appropriate way and then making \emph{almost sure} statements about the corresponding attractors.  This approach began with Falconer's seminal paper \cite{affine} in 1988, which introduced the \emph{affinity dimension} as a sure upper bound for the upper box dimension of any self-affine set and if the translation vectors are randomised and the norms of the defining matrices are strictly smaller than 1/2, then the Hausdorff, box and packing dimensions of the attractor are all almost surely equal to the affinity dimension. Some articles following this approach are \cite{solomyakmeasureanddim, falconergeneralized, shmerkinoverlapping, randomselfaffine, falconermiao, jordanjurga}. In contrast, the `specific approach' focuses on special classes of self-affine sets designed in a way to facilitate calculations and allows \emph{sure} statements to be made about the attractors.  This began with the work of Bedford \cite{bedford} and McMullen \cite{mcmullen} from 1985 which introduced self-affine carpets and computed their Hausdorff and box dimensions. Of particular note is that these values are typically different and strictly less than the affinity dimension. This second approach was further developed in \cite{lalley-gatz,baranski,fengaffine, fraserboxlike} among others. This paper has two main purposes. On one hand, we blend the two approaches in a natural context.  We begin with a Bedford-McMullen carpet and then randomise the translations whilst maintaining the key structural feature: the column alignment.  This will be elaborated on in the following section. On the other hand, we wanted to illustrate how a recent breakthrough of Hochman \cite{hochman} on the dimensions of self-similar sets and measures can also be applied to obtain analogous results for self-affine sets. We obtain formulae for the Hausdorff, box and packing dimensions valid outside of a small set of parameters, with two points of interest being that the values of the dimension are typically \emph{different} from each other and from the affinity dimension, and our class contains many \emph{overlapping} self-affine sets.

\subsection{Our setting}

Fix positive integers $n>m>1$ and divide the unit square into a uniform $m \times n$ grid.  The grid rectangles can now be labelled in a natural way as $D_0 = \{(i,j): i = 1, \dots, m, j=1, \dots, n\}$.  Choose a non-empty subset $D \subseteq D_0$ and for each $(i,j) \in D$, let $S_{(i,j)}$ denote the affine contraction which maps the unit square onto the rectangle indexed by $(i,j)$ defined by
\[
S_{(i,j)}(x,y) = \big(x/m+(i-1)/m, \ y/n + (j-1)/n \big).
\]
Together, the maps $\{S_{(i,j)}\}_{(i,j) \in D}$ form an iterated function system (IFS) and it is well-known that there exists a unique non-empty compact set $F$ satisfying
\[
F = \bigcup_{(i,j) \in D} S_{(i,j)}(F).
\]
This set $F$ is called the attractor of the IFS and this class of attractors was first studied in the mid-eighties independently by Bedford \cite{bedford} and McMullen \cite{mcmullen}, who each gave a formula for the Hausdorff and box dimensions.  Such sets are now known as Bedford-McMullen carpets.  We wish to consider the following generalisation.  For a given Bedford-McMullen system, we randomise the horizontal translates, whilst keeping the column structure intact, i.e., we always assume that if two rectangles are in the same column initially, then they are translated horizontally by the same amount.  See Figures \ref{fig:pattern} and \ref{fig:carpets}. The advantage of this approach is that because we keep some alignment in the construction, even though we randomise the system, the `typical' dimensions are still exceptional (we will see that in fact they are the same as the dimensions of the original system).  Thus we provide a smoothly parametrised family of potentially overlapping self-affine carpets whose dimensions are strictly less than the affinity dimension.

\begin{figure}[H]
	\centering
	\includegraphics[width=100mm]{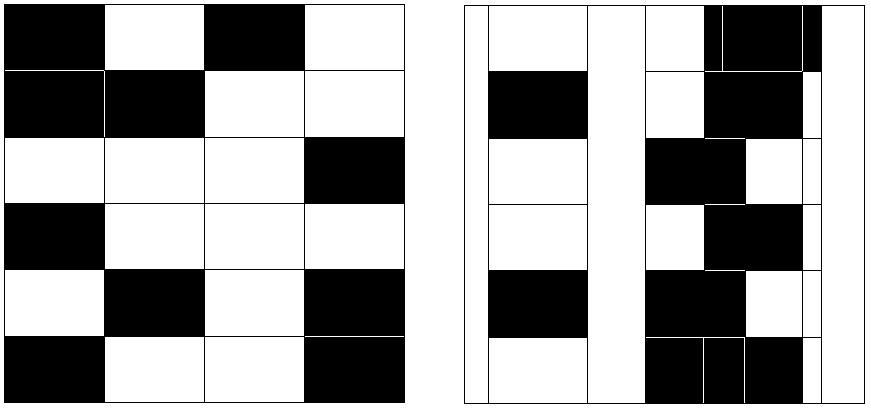}
\caption{A generating pattern for a Bedford-McMullen carpet (on the left) and a translation of the columns into an overlapping pattern (on the right).  In this case, $m=4$ and $n=6$.}  \label{fig:pattern}
\end{figure}

More formally, let $\ol{D} = \{i \in\{1,\dots, m\} : (i,j) \in D \text{ for some $j$}\}$ be the projection of $D$ onto the first co-ordinate.  To each $i \in \ol{D}$ we associate a `random translation'
$t_i \in [0,1-1/m]$ and for a given set of translates $\underline{t} = (t_i)_{i \in \ol{D}} \in [0,1-1/m]^{\ol{D}}$ we define a new IFS consisting of the maps
\[
S_{(i,j), \underline{t}}(x,y) \ = \ (x/m, y/n) \ + \ (t_i, (j-1)/n)
\]
and denote the attractor, which of course depends on $\underline{t}$, by $F_{\underline{t}}$.
In the case where $t_i = (i-1)/m$ for all $i \in \ol{D}$, then we recover the original Bedford-McMullen system. The restriction that $t_i\in [0,1-1/m]$ is meant to ensure the attractor is a subset of the unit square, and it is not essential.

We now wish to make statements about the dimensions of $F_{\underline{t}}$ in terms of the parameters $\underline{t} \in [0,1-1/m]^{\ol{D}}$.

\begin{figure}[H]
	\centering
	\includegraphics[width=100mm]{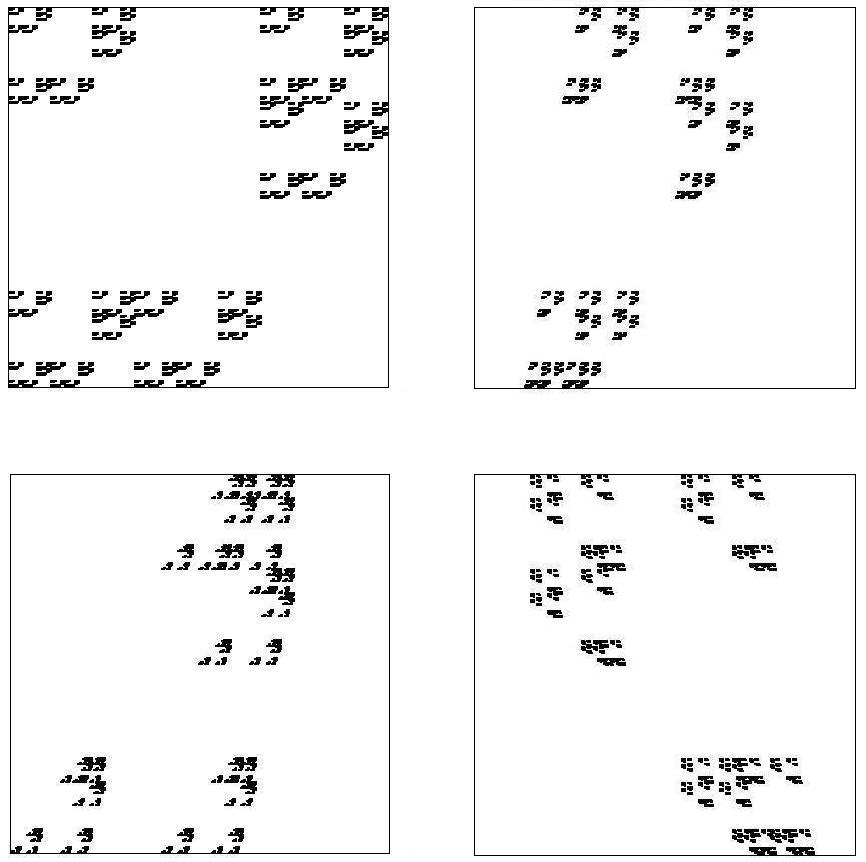}
\caption{Four self-affine carpets based on the same Bedford-McMullen system, but with the columns translated in different ways.  The original Bedford-McMullen carpet is on the top left.  In each case $m=3$ and $n=4$.} \label{fig:carpets}
\end{figure}

\section{Results}

In this section we state our main results on the dimensions of the self-affine sets described in the previous section.  We write $\dim_\H$, $\dim_\text{B}$ and $\dim_\P$ to denote the Hausdorff, box and packing dimensions respectively. Recall that an IFS $\{ S_1,\ldots, S_k\}$ is said to have an exact overlap if the semigroup generated by the $S_i$ is not free.   We write $N_i = \lvert\{j = 1, \dots, n: (i,j) \in D\}\rvert$ for the number of chosen rectangles in the $i$th column.

\begin{thm} \label{mainH} There exists a set $E\subset  [0,1-1/m]^{ \ol{D}}$ of Hausdorff and packing dimension $\lvert \ol{D}\rvert-1$ (in particular of zero $\lvert \ol{D}\rvert$-dimensional Lebesgue measure) such that
\begin{alignat*}{3}
 \dim_\text{\emph{H}} F_{\underline{t}} &= \frac{\log \sum_{i=1}^m N_i^{\log m/ \log n}}{\log m} && \quad\text{if }\underline{t}\in[0,1-1/m]^{ \ol{D}}\setminus E,\\
\dim_\text{\emph{H}} F_{\underline{t}} &< \frac{\log \sum_{i=1}^m N_i^{\log m/ \log n}}{\log m} && \quad\text{if }\underline{t}\in E.
\end{alignat*}
Moreover, if $\underline{t}$ is algebraic and the IFS $\{ x/m+t_i\}_{i\in \ol{D}}$ does not have an exact overlap, then $\underline{t}\notin E$.

\end{thm}

We will prove Theorem \ref{mainH} in Section \ref{proofH}. The main idea, which was inspired by results of Jordan \cite{jordan} and Jordan and Pollicott \cite{jordanpollicott}, is to use Marstrand's slicing theorem to bound the dimension from below by the sum of the dimension of the projection, and the dimension of a typical slice; see Section \ref{proofH} for further discussion.

\begin{thm} \label{mainB}
There exists a set $E\subset  [0,1-1/m]^{ \ol{D}}$ of Hausdorff and packing dimension $\lvert \ol{D}\rvert-1$ (in particular of zero $\lvert \ol{D}\rvert$-dimensional Lebesgue measure) such that
\begin{alignat*}{4}
 \dim_\text{\emph{B}} F_{\underline{t}} &=  \dim_\text{\emph{P}}  F_{\underline{t}} &= \frac{\log \lvert \ol{D} \rvert}{\log m} + \frac{\log \lvert D \rvert / \lvert \ol{D} \rvert}{\log n} & \quad\text{if }\underline{t}\in[0,1-1/m]^{ \ol{D}}\setminus E,\\
 \dim_\text{\emph{B}} F_{\underline{t}} &=  \dim_\text{\emph{P}} F_{\underline{t}} &<  \frac{\log \lvert \ol{D} \rvert}{\log m} + \frac{\log \lvert D \rvert / \lvert \ol{D} \rvert}{\log n} & \quad\text{if }\underline{t}\in E.
\end{alignat*}
Moreover, if $\underline{t}$ is algebraic and the IFS $\{ x/m+t_i\}_{i\in \ol{D}}$ does not have an exact overlap, then $\underline{t}\notin E$.
\end{thm}

We will prove Theorem \ref{mainB} in Section \ref{proofB}. Note that we say $\underline{t} = \{t_i\}_{i \in \ol{D}}$ is algebraic if all of the $t_i$ are algebraic.  A generalisation of our main results is discussed in Section \ref{subsecgeneralisation} below.  We note that the exceptional set $E$ in both cases is contained in the set for which ``super exponential concentration of cylinders'' occurs in the vertical projection; see Section \ref{sec:super-exponential} below.  However, we cannot guarantee that the exceptional set $E$ is precisely the same in both theorems.

We underline that the dimensions appearing in Theorems \ref{mainH} and \ref{mainB} are the same as the dimensions of the original carpet as proved by Bedford and McMullen. We recall that (for the unperturbed carpets) the box counting dimension is obtained by covering each rectangle in the $n$-th stage of the construction by the same number of disks of the appropriate size, independently of each other. When allowing covers by disks of different sizes, it is usually more efficient to cover large collections of parallel rectangles by a single disk, leading to the expression for Hausdorff dimension (which also has a variational interpretation as the supremum of Hausdorff dimensions of Bernoulli measures for the natural Markov partition of the carpet). Our results suggest that this geometric picture typically persists when the columns are allowed to overlap.

We finish this section by commenting on the relation between these results and previous work on the subject. The ``hybrid'' approach to the dimension of self-affine sets was undertaken previously in \cite{jordanpollicott, barany}. These papers give a formula, valid for Lebesgue almost all parameters, for certain (different) families of parametrized self-affine carpets. A point in common with our model is that these typical dimensions are strictly less than the affinity dimension. However, in \cite{jordanpollicott} only Hausdorff dimension was considered, while the results of \cite{barany} are for box-counting dimension and involve non-overlapping self-affine sets. Also, because we use Hochman's recent results (which were not available to the authors of \cite{jordanpollicott, barany}), we obtain far better information about the exceptional set.

\section{Symbolic notation}

Most of the proofs in the subsequent sections are symbolic in nature, and thus rely more on the combinatorics of the symbolic spaces $D^{\mathbb{N}}$ and $\ol{D}^{\mathbb{N}}$ than the geometry of the corresponding fractals. In Section \ref{sec:super-exponential} we review the results of Hochman which will allow us to pass from the symbolic information back to the geometric setting. In this section we briefly summarise some notation we will use throughout the rest of the paper.

Define the vertical projection $\pi:[0,1]^2 \to [0,1]$ by $\pi(x,y) = x$. As is often the case for self-affine carpets, this projection will play a key role. For $i\in\ol{D}$ and $\underline{t}\in [0,1-1/m]^{\ol{D}}$, we will denote $\ol{S}_{i,\underline{t}}(x)=\tfrac{1}{m}x+t_i$. Note that the IFS $\{ \ol{S}_{i,\underline{t}}\}_{i\in\ol{D}}$ generates the projection $\pi F_{\underline{t}}$.

Given $\lambda=(\lambda_1,\ldots,\lambda_k) = ((i_1,j_1),\ldots,(i_k,j_k))\in D^k$, and $\rho=(i_1,\ldots,i_k)\in\ol{D}^k$, we denote
\begin{align*}
S_{\lambda, \underline{t}} &= S_{(i_1,j_1), \underline{t}} \circ \cdots \circ S_{(i_k,j_k), \underline{t}},\\
\ol{S}_{\rho,\underline{t}} &= \ol{S}_{i_1,\underline{t}} \circ\cdots\circ \ol{S}_{i_k,\underline{t}}.
\end{align*}

Rather than define all the notions of dimension we are interested in, we simply refer the reader to \cite[Chapters 2-3]{falconer}.  The key properties of Hausdorff and packing dimensions which we need will be discussed in the following sections when required.  We now recall the definition of box dimension.  The lower and upper box dimensions of a bounded set $F \subseteq \mathbb{R}^n$ are defined by
\[
\underline{\dim}_\text{B} F = \liminf_{r \to 0} \, \frac{\log N_r (F)}{-\log r}
\qquad
\text{and}
\qquad
\overline{\dim}_\text{B} F = \limsup_{r \to 0} \,  \frac{\log N_r (F)}{-\log r},
\]
respectively, where $N_r (F)$ is the smallest number of sets required for a $r$-cover of $F$. Here an $r$-cover of $F$ is a finite (or countable) collection of open sets $\{U_k\}_k$ with the property that the diameter of each set $\lvert U_k \rvert \leq r$ and $F \subseteq \cup_k U_k$. If $\underline{\dim}_\text{B} F = \overline{\dim}_\text{B} F$, then we call the common value the box dimension of $F$ and denote it by $\dim_\text{B} F$.  It is useful to note that we can replace $N_r$ with several different definitions all based on covering or packing the set at scale $r$, see \cite[Section 3.1]{falconer}.  For example, it can be the number of cubes in an $r$-grid which intersect $F$.

\section{Super-exponential concentration of cylinders and the dimensions of self-similar sets and measures}
\label{sec:super-exponential}

In this section we recall a recent result of Hochman \cite{hochman} on the dimensions of self-similar measures. We consider only the special case of \emph{homogeneous} self-similar measures, which is all we will require. To this end, consider an affine IFS of the form $\mathcal{I}=\{ S_i(x)=a x+t_i\}_{i\in A}$, where $A$ is a finite index set and $a \in (0,1)$ is a fixed contraction rate and the maps $S_i$ are assumed to act on $\mathbb{R}$.

\begin{defn}
We say that the IFS $\mathcal{I}=\{ S_i(x)= a x+t_i\}_{i\in A}$ has \emph{super-exponential concentration of cylinders} if $-\log\Delta_k/k\to\infty$ (with the convention $\log 0=-\infty$), where
\[
\Delta_k =  \min_{\rho\neq \rho'\in A^k} | S_\rho(0)-S_{\rho'}(0)|,
\]
and, as usual, $S_\rho(x) =  S_{i_1}\circ\cdots\circ S_{i_k}$ if $\rho=(i_1,\ldots,i_k) \in A^k$.
\end{defn}

In other words, $\mathcal{I}$ has super-exponential concentration of cylinders if the distance between cylinders of level $k$ coming from different codes decreases faster than any power as a function of $k$. The only known mechanism by which super-exponential concentration of cylinders can occur is the presence of exact overlaps; by definition, this means that $\Delta_k=0$ for some $k$. One observation that will be useful later is that if $\mathcal{I}$ does not have super-exponential concentration of cylinders, then the same is true for any IFS which is obtained by first iterating all the maps in $\mathcal{I}$ a fixed number of times, and then dropping some of the maps.

The following is the key result of Hochman that we will require. Recall that the Hausdorff dimension of a Borel probability measure $\nu$ is defined by
\[
\dim_\H \nu = \inf \{ \dim_\H F :  \text{ $F$ is a Borel set with $\nu(F) = 1$ } \}.
\]

\begin{thm}[{\cite[Theorem 1.1]{hochman}}] \label{thm:hochman}
Suppose the IFS $\mathcal{I}=\{ S_i(x)= a x+t_i\}_{i\in A}$ does \emph{not} have super-exponential concentration of cylinders. Let $p=(p_i)_{i\in A}$ be a probability vector, and let $\nu$ be the self-similar measure associated to the IFS $\mathcal{I}$ and the vector $p$, that is, the unique Borel probability measure satisfying
\[
\nu = \sum_{i\in A} p_i\, \nu\circ S_i^{-1},
\]
Then
\[
\dim_\text{\emph{H}} \nu = \min\left(\frac{\sum_{i\in A} p_i \log p_i}{\log a},1\right).
\]
In particular, if $F$ is the invariant set for the IFS $\mathcal{I}$, that is, the only nonempty compact set satisfying $F=\bigcup_{i\in A} S_i(F)$, then
\[
\dim_\text{\emph{H}} F= \min\left(\frac{\log |A|}{\log(1/a)},1\right).
\]
\end{thm}

The next result, which also follows from Hochman's work, tells us that super-exponential concentration of cylinders is a rare phenomenon - in a quantitative sense and in some special cases, as rare as exact overlaps.

\begin{prop} \label{rare-secc}
Let $A$ be a finite index set and fix $a\in (0,1/2)$.
\begin{enumerate}
\item The family of $(t_i)_{i\in A}$ such that the IFS $\mathcal{I}=\{ ax+t_i\}_{i\in A}$ has super-exponential concentration of cylinders has Hausdorff and packing dimension $|A|-1$.
\item If $a$ and $t_i, (i\in A)$, are all algebraic, then the IFS $\mathcal{I}=\{ ax+t_i\}_{i\in A}$ has super-exponential concentration of cylinders if and only if there is an exact overlap, that is, if and only if $\Delta_k=0$ for some $k$.
\end{enumerate}
\end{prop}
\begin{proof}
Let $E_0$ denote the set of parameters in a box $[-M,M]^{\lvert A \rvert}$ for which there is super-exponential concentration of cylinders. If $i\neq j\in A$, then $E_0$ contains the piece of hyperplane $\{ t_i=t_j\}\cap [-M,M]^{\lvert A \rvert}$, so $\dim_\H(E_0)\ge |A|-1$. The proof of the upper bound is similar to the proof of \cite[Theorem 1.8]{hochman} (in fact simpler because of the linearity of the projection map) and also follows from results in \cite{hochman2}, but we give the proof for completeness. Given a set of translations $\underline{t} = (t_i)_{i\in A}$, we write the maps in the corresponding IFS as $S_{i,\ut}=a x+t_i$ to emphasise dependence on $\underline{t}$.  Given two sequences $\rho = (i_1, \dots, i_n), \rho' = (j_1, \dots, j_n) \in A^n$, let $\Delta_{\rho, \rho'}: [-M,M]^{\lvert A \rvert}\to \mathbb{R}$ be the map
\[
\Delta_{\rho, \rho'}(\underline{t}) = S_{\rho,\ut}(0)-S_{\rho',\ut}(0),
\]
where $S_{\rho,\ut},S_{\rho',\ut}$ are the compositions of maps coming from the IFS $\{S_{i,\ut}=a x+t_i\}_{i\in A}$. It follows from the definition of super-exponential concentration that
\[
E_0 = \bigcap_{\varepsilon>0} \bigcup_{N=1}^\infty \bigcap_{n>N} \bigcup_{\rho \neq \rho'\in A^n} \Delta_{\rho,\rho'}^{-1}(-\e^n,\e^n).
\]
Since $\ut\to S_{\rho,\ut}(0)$ is linear, then so is $\Delta_{\rho,\rho'}$ for any $\rho,\rho' \in A^n$. If we write $\Delta_{\rho,\rho'}(\ut)= \sum_{\ell\in A} c_\ell t_\ell$, then the coefficients $c_\ell$ are given by
\[
c_\ell = \sum_{k\in \{1,\ldots,n\}: i_k=\ell} a^{k-1} - \sum_{k\in\{1,\ldots,n\}: j_k=\ell} a^{k-1}.
\]
Hence $|c_\ell|\le 1/(1-a)$ for all $\ell$ and, if $\rho \neq \rho'$, there is $\ell$ such that $|c_\ell|\ge a^n(1-a/(1-a))$; this is positive since $a\in (0,1/2)$. These considerations imply that for any distinct $\rho,\rho' \in A^n$, the set $\Delta_{\rho,\rho'}^{-1}(-\e^n,\e^n)$ is contained in the $C_a(\e/a)^n$-neighborhood of the hyperplane $\Delta_{\rho,\rho'}=0$, where $C_a$ is a constant independent of $n$ and $\e$. Hence $\cup_{\rho \neq \rho'\in A^n} \Delta_{\rho,\rho'}^{-1}(-\e^n,\e^n)$ can be covered by $C_{a,|A|,M}|A|^{2n} (a/\e)^{n(|A|-1)}$ balls of radius $(\e/a)^n$, where $C_{a,|A|,M}$ is a constant depending only on $a,|A|$ and $M$, and in particular, not on $n$ or $\e$. It follows that for $N \geq 1$
\[
\ol{\dim}_{\text{B}}\left(\bigcap_{n>N} \bigcup_{\rho \neq \rho' \in A^n} \Delta_{\rho,\rho'}^{-1}(-\e^n,\e^n)\right) \le |A|-1+\frac{2\log|A|}{\log(a/\e)}.
\]
We recall that one characterization of the packing dimension of a set $E$ is $\dim_{\P}(E) = \inf\{\sup_i\ol{\dim}_{\text{B}}(E_i): E\subseteq \cup_i E_i\}$, see \cite[Proposition 3.8]{falconer}. Therefore the above implies that
\begin{align*}
\dim_\P E_0 &\le \lim_{\e\searrow 0} \dim_\P\left( \bigcup_{N=1}^\infty \bigcap_{n>N} \bigcup_{\rho \neq \rho'\in A^n} \Delta_{\rho,\rho'}^{-1}(-\e^n,\e^n)\right) \\
&\le \lim_{\e\searrow 0}  \left( |A|-1+\frac{2\log|A|}{\log(a/\e)} \right) = |A|-1,
\end{align*}
which yields the first claim since $M$ was arbitrary.

The second part of the proposition is \cite[Theorem 1.5]{hochman}.
\end{proof}

\section{Calculation of the Hausdorff dimension} \label{proofH}

In this section we prove Theorem \ref{mainH}, which gives the Hausdorff dimension of $F_{\underline{t}}$ outside of a small exceptional set of $\underline{t}$.  The proof relies on Marstrand's slice theorem and being able to control two things: the Hausdorff dimension of a particular self-similar measure supported on $\pi(F_{\underline{t}})$; and, for this measure, the almost sure Hausdorff dimension of the vertical slices through points in $\pi(F_{\underline{t}})$.  We borrow the slicing idea from the works of Jordan \cite{jordan} and Jordan and Pollicott \cite{jordanpollicott}, where the Hausdorff dimension of overlapping Sierpi\'nski gaskets and carpets was considered.  In particular, Jordan and Pollicott \cite[Section 6.2]{jordanpollicott} find the Hausdorff dimension of certain overlapping carpets of Bedford-McMullen type, for almost all values of the parameter in a certain interval; however, in their work the parameter determines the contraction ratios, while we work with fixed contractions and vary the translations.

Throughout this section let
\[
s = \frac{\log \sum_{i=1}^m N_i^{\log m/ \log n}}{\log m}
\]
be the target Hausdorff dimension.

Let $\mu$ be the McMullen measure on the symbolic space $D^\mathbb{N}$, i.e., the Bernoulli measure with weights
\[
p_{(i,j)} = N_i^{\log m/ \log n-1}/m^s
\]
and let $\ol{\mu} = \mu \circ \pi^{-1}$ be the natural projection of $\mu$ onto $\ol{D}^\mathbb{N}$, where in a slight abuse of notation we let $\pi$ denote projection onto the first coordinate in both the symbolic and geometric spaces.  In particular, $\ol{\mu}$ is a Bernoulli measure with weights
\[
p_{i} =  N_i^{\log m/ \log n}/m^s.
\]
Let $\Pi_{\underline{t}}$ denote the natural coding map from $D^\mathbb{N}$ to $F_{\underline{t}}$ and $\ol{\Pi}_{\underline{t}}$ denote the natural coding map from $\ol{D}^\mathbb{N}$ to $\pi(F_{\underline{t}})$, the projection of $F_{\underline{t}}$ onto the horizontal axis. Note that $\ol{\mu} \circ \ol{\Pi}_{\underline{t}}^{-1}$ is nothing else than the self-similar measure for the IFS $\{ \ol{S}_{i,\underline{t}}  \}_{i\in \ol{D}}$ with weights $(p_i)_{i\in \ol{D}}$. The following is then immediate from Theorem \ref{thm:hochman} and Proposition \ref{rare-secc}.

\begin{lma} \label{keylem1}

Suppose $m\ge 3$. Let $E_0$ be the set of parameters $\underline{t}\in  [0,1-1/m]^{\ol{D}}$ such that the IFS $\{ \ol{S}_{i,\underline{t}}\}_{i\in \ol{D}}$ has super-exponential concentration of cylinders. Then $E_0$ has Hausdorff and packing dimension $|\ol{D}|-1$. Moreover, if $\underline{t}\in [0,1-1/m]^{\ol{D}}\setminus E_0$, then
\begin{equation} \label{keylem1eq}
\dim_\text{\emph{H}} \big( \ol{\mu} \circ \ol{\Pi}_{\underline{t}}^{-1} \big) \ =  \ - \frac{\sum_{i \in \ol{D}} p_i \log p_i}{\log m}
\end{equation}
Furthermore, if $\underline{t}$ is algebraic and the IFS $\{ \ol{S}_{i,\underline{t}}\}_{i\in \ol{D}}$  does not have an exact overlap, then $\underline{t}\notin E_0$.
\end{lma}

For $x \in [0,1]$ let $L_x = \{(x,y) : y \in \mathbb{R}\}$ be the vertical line through the point $(x,0)$.

\begin{lma}  \label{keylem2}
For any $\underline{t} \in [0,1-1/m]^{\ol{D}}$ we have
\[
\dim_\text{\emph{H}} L_x \cap F_{\underline{t}} \ \geq \ \sum_{i \in \ol{D}} \, p_i \,  \frac{\log N_i}{\log n}
\]
for $ \ol{\mu} \circ \ol{\Pi}_{\underline{t}}^{-1}$ almost all $x \in \pi(F_{\underline{t}})$.
\end{lma}

\begin{proof}
Let $x  = \ol{\Pi}_{\underline{t}}(\mathbf{i}) \in \pi(F_{\underline{t}})$ for some $\mathbf{i} = (i_1, i_2, \dots) \in \ol{D}^{\mathbb{N}}$.  It is straightforward to see that $L_x \cap F_{\underline{t}}$ contains the set
\[
\bigcap_{k \in \mathbb{N}} \bigcup_{\substack{j_1 : (i_1,j_1) \in D \\ \vdots \\ j_k : (i_k,j_k) \in D}} S_{(i_1,j_1), \underline{t}} \circ \cdots \circ S_{(i_k,j_k), \underline{t}} \big( [0,1]^2\big) \cap L_x
\]
which is a particular realisation of a 1-variable random self-similar set in the sense of Barnsley-Hutchinson-Stenflo (see e.g. \cite{vvariable}) where the deterministic IFSs used are the natural IFSs of similarities induced by the columns in our construction.  A realisation of a 1-variable random constructions corresponds to a particular infinite sequence over the set of deterministic IFSs and the above example is given by the sequence $\mathbf{i}$.  We can apply the dimension results in \cite[Section 4]{vvariable} with weights $\{p_i\}$ to obtain that such 1-variable random self-similar sets have dimension
\[
\sum_{i \in \ol{D}} \, p_i \,  \frac{\log N_i}{\log n}
\]
for $\ol \mu$ almost all $\mathbf{i} \in \ol{D}^{\mathbb{N}}$, which completes the proof.
\end{proof}

We will use the following version of Marstrand's slice theorem, which follows from, for example, \cite[Corollary 7.12]{falconer}.

\begin{lma} \label{mst}
Let $F \subseteq \mathbb{R}^2$ and let $\nu$ be a Borel probability measure with support in $\mathbb{R}$.  If $\dim_\text{\emph{H}} (F \cap L_x) \geq s$ for $\nu$ almost all $x$, then $\dim_\text{\emph{H}} F \geq s + \dim_\text{\emph{H}} \nu$.
\end{lma}

We can now complete the proof of Theorem \ref{mainH}.

\begin{proof}
In McMullen's original proof, the calculation of the upper bound of the Hausdorff dimension of $F_{\underline{t}}$ is performed using covers by approximate squares on an appropriately defined symbolic space. Since these covers cover each symbolic `column' independently, the upper bound continues to hold when projecting to the actual fractal even in the presence of overlaps. Thus $\dim_{\text{H}}(F_{\underline{t}})\le s$ for all $t\in
[0,1-1/m]^{\ol{D}}$.

Now suppose $t_{i_1}=t_{i_2}$ for some distinct $i_1,i_2\in\ol{D}$, i.e. we have an exact column overlap. Symbolically, this corresponds to replacing two columns with $N_{i_1}, N_{i_2}$ rectangles by a single column with $N' \le N_{i_1}+N_{i_2}$ rectangles. Then in this case we get an upper bound
\[
\dim_{\text{H}}(F_{\underline{t}}) \le \frac{\log \left((N_{i_1}+N_{i_2})^{\log m/\log n}+\sum_{i\in\{1,\ldots,m\}\setminus\{i_1,i_2\}} N_i^{\log m/ \log n}\right)}{\log m} < s,
\]
using that $(N_{i_1}+N_{i_2})^{\gamma} < N_{i_1}^\gamma+N_{i_2}^\gamma$ for $\gamma\in(0,1)$. It follows that $\dim_{\text{H}}(E)\ge \lvert \ol{D}\rvert-1$.

We now deal with the lower bound for $\dim_{\text{H}}(F_{\underline{t}})$. The case $m=2$ is not very interesting, as the systems obtained for any $t_1\neq t_2$ are affinely conjugated to each other, and hence have the same dimensions as the original carpet. Hence from now on we assume $m\ge 3$.  In this situation the exceptional set $E$ in the theorem can be taken to be precisely $E_0$, where $E_0$ is the $(\lvert\ol{D}\rvert-1)$-dimensional set from Lemma \ref{keylem1}.  Fix $\underline{t}\in [0,1-1/m]^{\ol{D}}\setminus E_0$.  It is enough to show that $\dim_{\text{H}}(F_{\underline{t}}) \ge s$. Lemmas \ref{keylem1}-\ref{keylem2} and Marstrand's slice theorem (Lemma \ref{mst}) combine to yield
\begin{eqnarray*}
\dim_\text{H} F_{\underline{t}} & \geq & - \frac{\sum_{i \in \ol{D}} p_i \log p_i}{\log m} \ + \ \sum_{i \in \ol{D}} \, p_i \,  \frac{\log N_i}{\log n} \\ \\
& = & \sum_{i \in \ol{D}} p_i  \bigg( \frac{\log N_i}{\log n} \ - \ \frac{\log p_i}{\log m} \bigg) \\ \\
& = & \sum_{i \in \ol{D}} p_i  s \qquad \qquad \text{by the definition of $p_i$}\\ \\
& = & s
\end{eqnarray*}
as required.
\end{proof}

\section{Calculation of the box and packing dimensions} \label{proofB}

In this section we prove Theorem \ref{mainB} which gives the packing and box dimensions of $F_{\underline{t}}$ outside of a small set of exceptional $\underline{t}$.  This proof is rather more complicated, but perhaps less elegant, than the Hausdorff dimension case given in the previous section.  The reason for this is that we do not have a useful analogue of Marstrand's slice theorem and we do not have an analogue of the McMullen measure, i.e., a Bernoulli measure with full \emph{packing} dimension.

The box (and packing) dimension of $F_{\underline{t}}$ depends on three things: the dimension of the projection $\pi(F_{\underline{t}})$; the number of maps in the IFS; and how much `separation' there is in the construction.  In order to find $\underline{t}$ which give rise to maximal box dimension, these three things have to be controlled, and optimised, simultaneously.  Our strategy is somewhat involved and so we briefly describe it here before we begin the proof.   First we apply Hochman's results to control the dimension of $\pi(F_{\underline{t}})$.  For a fixed $\underline{t}$ which maximises $ \dim_\text{B} \pi(F_{\underline{t}})$, the defining IFS has the correct projection dimension and enough maps but not enough separation and so we need to `approximate it from within' by finding a subsystem which has almost enough maps and enough separation to give the desired result.  We could find a subsystem of the \emph{projected} IFS which gives the same projection dimension and guarantees separation, but this introduces a problem: there will be too few maps in the induced IFS on the square if the original system did not have uniform vertical fibres (the system is said to have uniform vertical fibers if the numbers $N_i=|\{ j: (i,j)\in D\}|$ are constant over $i\in \ol{D}$).  As such, we employ a technique similar to that used in \cite{ferguson_proj} by finding a subsystem of the IFS on the square which has almost the correct number of mappings, but uniform vertical fibres.  The issue now is that the projected dimension may be too small, but we can nevertheless find a subsystem of the projected IFS with the same (albeit too small) dimension which guarantees separation in the induced IFS on the square.  Instead of treating the induced subsystem as an IFS in its own right, we consider images of the original overlapping self-affine set by these maps.  This means that when we come to cover the, now disjoint, images, we are covering a subset of $F_{\underline{t}}$ with the correct projection dimension and, because we have uniform fibres, enough maps.

Throughout this section let
\[
\ol{s} \ = \ \frac{\log \lvert \ol{D} \rvert}{\log m} \qquad \text{and} \qquad s \  = \ \frac{\log \lvert \ol{D} \rvert}{\log m} + \frac{\log \lvert D \rvert / \lvert \ol{D} \rvert}{\log n}.
\]
In particular, $s$ is the target almost sure box dimension of $F_{\underline{t}}$ and $\ol{s}$ is the target almost sure box dimension of the relevant projection $F_{\underline{t}}$, which plays a key role.  We will prove Theorem \ref{mainB} in the box dimension case and note that, since each $F_{\underline{t}}$ is compact and has the property that every open ball centered in $F_{\underline{t}}$ contains a bi-Lipschitz image of $F_{\underline{t}}$, $\dim_\P F_{\underline{t}} = \overline{\dim}_\text{B} F_{\underline{t}}$ for all ${\underline{t}}$.  For more details on this useful alternative formulation of packing dimension, see \cite[Corollary 3.9]{falconer}.

We note the following consequence of Theorem \ref{thm:hochman} and Proposition \ref{rare-secc}.

\begin{lma} \label{boxlem1}
Fix $m\ge 3$. Let $E_0$ be the set of parameters $\underline{t}\in  [0,1-1/m]^{\ol{D}}$ such that the IFS $\{ \ol{S}_{i,\underline{t}} \}_{i\in \ol{D}}$ has super-exponential concentration of cylinders. Then $E_0$ has Hausdorff and packing dimension $|\ol{D}|-1$. Moreover, if $\underline{t}\in [0,1-1/m]^{\ol{D}}\setminus E_0$, then
\[
\dim_\text{\emph{H}} \pi(F_{\underline{t}}) \ = \ \dim_\text{\emph{B}} \pi(F_{\underline{t}}) \ = \ \frac{\log \lvert \ol{D} \rvert}{\log m} = \ol{s}
\]
Furthermore, if $\underline{t}$ is algebraic and the IFS $\{ \ol{S}_{i,\underline{t}} \}_{i\in \ol{D}}$ does not have an exact overlap, then $\underline{t}\notin E_0$.
\end{lma}

Let $N = \lvert D \rvert$,
\[
p= (1/m)^{\ol{s}} (1/n)^{s-\ol{s}} = 1/N
\]
and, for $k \in \mathbb{N}$, let
\[
\theta(k) = \sum_{i \in D} \lfloor p k \rfloor  = N \lfloor k/N \rfloor \in \mathbb{N}.
\]
Note that $k-N \leq \theta(k) \leq k$ for all $k \in \mathbb{N}$.  Consider $D^{\theta(k)}$ and let
\[
H_k = \Big\{ \lambda = (\lambda_1, \dots, \lambda_{\theta(k)}) \in D^{\theta(k)} : \text{for all }(i,j) \in D, \,  \lvert\{n \in \{1, \dots, \theta(k)\} : \lambda_n =(i,j)\} \rvert = \lfloor p k \rfloor \Big\}.
\]
It is straightforward to see that
\begin{equation} \label{combina}
\lvert H_k \rvert = \frac{\theta(k)!}{\prod_{i \in D} \lfloor pk \rfloor !} = \frac{\left( N \lfloor k/N \rfloor\right)!}{ \left( \lfloor k/N \rfloor! \right)^N}
\end{equation}
and the IFS $\{ S_{\lambda,\overline{t}}\}_{\lambda\in H_k}$ corresponding to $H_k$ has uniform vertical fibres.  Define $s_k$ by
\[
s_k = \frac{\log \lvert H_k \rvert }{k \log n } \ +  \ \ol{s}\bigg(1-\frac{\log m}{\log n } \bigg)
\]

\begin{lma} \label{skconverges}
We have $s_k \to s$ as $k \to \infty$.
\end{lma}

\begin{proof}
We will use a version of Stirling's approximation for the logarithm of large factorials.  This states that for all $b \in \mathbb{N} \setminus \{1\}$ we have
\begin{equation} \label{stirling}
b \log b - b  \ \leq \ \log b! \ \leq \  b \log b - b +\log b.
\end{equation}
Note that
\[
s - \ol{s}\left(1-\frac{\log m}{\log n}\right) = \frac{\log N}{\log n}.
\]
Hence we need to show that
\[
 \frac{\log \lvert H_k \rvert }{k  } \to \log N \quad\text{as }k\to\infty.
\]
It is easy to see that $\lvert H_k \rvert \le N^k$ for all $k$. For the opposite inequality, we estimate, for large enough $k$,
\begin{align*}
\frac{\log \lvert H_k \rvert }{k} &= \frac{\log \theta(k)! - \sum_{i \in D} \log  \lfloor pk \rfloor!}{k}&\\
&\ge \frac{\theta(k)\log \theta(k) - \theta(k)  - \sum_{i \in D}\Big( \lfloor pk \rfloor \log  \lfloor pk \rfloor - \lfloor pk \rfloor + \log \lfloor pk \rfloor \Big)}{k}& \text{by }\eqref{stirling}\\
&= \frac{\theta(k)\log \theta(k) - \sum_{i \in D} \lfloor pk \rfloor \log  \lfloor pk \rfloor }{k}   -   \frac{ N \log \lfloor pk \rfloor }{k}&\\
&\ge \frac{\theta(k)\log \theta(k) - \log(pk)\sum_{i \in D} \lfloor pk \rfloor}{k}   -   \frac{ N \log \lfloor pk \rfloor }{k}&\\
&= \frac{\theta(k)\log \left(\theta(k)/(pk)\right)}{k}  -   \frac{ N \log \lfloor pk \rfloor }{k}&\\
&\ge \frac{k-N}{k}\log \left(\frac{k-N}{pk}\right)  -   \frac{ N \log \lfloor pk \rfloor }{k}\\
&\to \log N \quad\text{as }k\to\infty,
\end{align*}
which completes the proof.

\end{proof}

Let $\varepsilon \in (0, \ol{s})$ and fix $k \in \mathbb{N}$ large enough to guarantee that $s_k \geq s-\varepsilon$ which we can do by Lemma \ref{skconverges}.  Let
\[
\ol{H}_k =  \{ (i_1,\ldots,i_{\theta(k)}): \left((i_1,j_1),\ldots, (i_{\theta(k)},j_{\theta(k)})\right)\in H_k \text{ for some } j_1,\ldots,j_{\theta(k)}\},
\]
and consider the IFS of similarities $\mathcal{I}_k = \{  \ol{S}_{i,\underline{t}}\}_{i\in \ol{H}_k}$ associated to $\ol{H}_k$. Since the original projected IFS $\{ \ol{S}_{i,\underline{t}}\}_{i\in\ol{D}}$  had no super-exponential concentration of cylinders by assumption, neither does $\mathcal{I}_k$, and so, by Theorem \ref{thm:hochman}, the attractor has Hausdorff and box dimension equal to the new similarity dimension, given by
\begin{equation} \label{newsimdim}
\ol{s}_k = \frac{\log \lvert \ol{H}_k \rvert}{k \log m}.
\end{equation}
The following lemma is a version of a standard result which allows one to approximate the dimension of a self-similar set with overlaps by subsystems without overlaps.  Recall that an IFS $\{S_i\}_{i \in A'}$ with attractor $F$ satisfies the \emph{strong separation condition} (SSC) if the $S_{i}(F) \cap S_{i'}(F) = \emptyset$ for distinct $i, i' \in A'$.  If the SSC is satisfied it makes the IFS and corresponding attractor much easier to handle.

\begin{lma} \label{approxwithin}
Let $\{S_i\}_{i \in A}$ be an IFS of similarities on $[0,1]$, each with the same contraction ratio $a \in (0,1)$, and with self-similar attractor $F$ having Hausdorff and box dimension $t$ and let $\varepsilon>0$.  There exists $\ell_0 \in \mathbb{N}$ such that for all $\ell \geq \ell_0$ there exists a subsystem corresponding to a subset $A_{\ell} \subseteq A^\ell$ which satisfies the SSC and
\[
\lvert A_{\ell} \rvert \geq 3^{-t} a^{-\ell(t-\varepsilon)}.
\]
\end{lma}
Before proving this lemma, we note that the $\ell$ appearing in  $A_{\ell}$ merely indicates dependence on $\ell$, whereas the $\ell$ appearing in  $A^{\ell}$ indicates, as usual, that we consider words of length $\ell$ over $A$.

\begin{proof}
This follows easily from the Vitali covering lemma, which has been used to prove a similar result previously, see, for example, \cite{orponen}.  We include the details for completeness. Let $\ell \in \mathbb{N}$ and consider the set $A^l$ consisting of words of length $l$ over $A$ and the sets $\{S_i([0,1])\}_{i \in A^\ell}$. By the Vitali covering lemma, we can extract a subset $\{S_i([0,1])\}_{i \in A_{\ell}}$ for some $A_\ell \subseteq A^\ell$ such that
\[
F \subseteq \bigcup_{i \in A^\ell} S_i([0,1]) \subseteq   \bigcup_{i \in A_{\ell}} S_i([-1,2])
\]
and where the sets $\{ S_i([0,1])\}_{i \in A_{\ell}}$ are pairwise disjoint subsets of $[0,1]$, which means that the IFS induced by $A_{\ell}$ satisfies the SSC.  It follows that $N_{3a^\ell}(F) \leq \lvert A_{\ell} \rvert$.  Moreover, the definition of box dimension implies that for all $\varepsilon>0$ there exists $\ell_0 \in \mathbb{N}$ such that for all $\ell \geq \ell_0$,
\[
N_{3a^\ell}(F) \geq \big(3a^\ell\big)^{-(t-\varepsilon)}
\]
which completes the proof.
\end{proof}

We can now complete the proof of Theorem \ref{mainB}.

\begin{proof}
The upper bound $\overline{\dim}_{\text{B}} F_{\underline{t}}\le s$ holds for all $\underline{t}\in [1-1/m]^{\ol{D}}$; this follows from \cite[Theorem 2.4]{fraserboxlike} which gave an upper bound for the upper box dimension of a class of self-affine carpets (which contains all of the sets $F_{\underline{t}}$) in terms of the box dimensions of the orthogonal projections without any separation conditions. For completeness, we sketch the argument in this case. Since $F_{\underline{t}} = \cup_{\lambda\in D^{\ell}} S_{\lambda,\underline{t}} F_{\underline{t}}$, we have
\[
N_r(F_{\underline{t}}) \le \sum_{\lambda\in D^{\ell}} N_r(S_{\lambda,\underline{t}} F_{\underline{t}}).
\]
Let $r=(1/n)^\ell$. Since $S_{\lambda,\underline{t}}$ maps the unit square to a rectangle of size $(1/m)^{\ell}\times r$, it follows that
\[
N_r(F_{\underline{t}})  \le C |D|^{\ell} N_{r m^{\ell}}(\pi F_{\underline{t}})
\]
where $C>0$ is a constant which does not depend on $r,l,m$ or $n$.  But $\pi F_{\underline{t}}$ is a self-similar set with similarity dimension $\ol{s}=\log|\ol{D}|/\log m$. As the upper box counting dimension is bounded above by the similarity dimension, for any $\varepsilon>0$ there is $C_\varepsilon>0$ such that $N_{\rho}(\pi F_{\underline{t}}) \le C_{\varepsilon} \rho^{\ol{s}-\varepsilon}$ for all $\rho>0$. Applying this with $\rho= r m^{\ell}$ and putting all estimates together yields the desired upper bound.

Now, as in the proof of Theorem \ref{mainH}, pick $\underline{t}$ such that $t_{i_1}=t_{i_2}$ for some distinct $i_1,i_2\in \ol{D}$. This merges two columns and does not increase the total number of rectangles, so applying the upper bound to the resulting system we get
\[
\overline{\dim}_{\text{B}} F_{\underline{t}} \le \frac{\log( \lvert \ol{D} \rvert-1)}{\log m} + \frac{\log \lvert D \rvert / (\lvert \ol{D} \rvert-1)}{\log n} < s,
\]
since $m<n$. It follows that $\dim_\H E\ge \lvert \ol{D} \rvert -1$.

As in the proof of Theorem \ref{mainH}, the lower bound in the case $m=2$ is straightforward, so we assume that $m\ge 3$. Again, in this setting the exceptional set $E$ in the theorem can be taken to be precisely $E_0$, where $E_0$ is the $(\lvert\ol{D}\rvert-1)$-dimensional set from Lemma \ref{boxlem1}. Fix a $\underline{t} \in [0,1-1/m]^{\ol{D}}\setminus E_0$.  For this $\underline{t}$ we will prove that the lower box dimension is at least $s$, which completes the proof.

We will apply Lemma \ref{approxwithin} to the IFS of similarities $\mathcal{I}_k$ corresponding to $\ol{H}_k$.  In particular, there exists $\ell_0 \in \mathbb{N}$ such that for all $\ell \geq \ell_0$  we may find a subset
\[
\ol{G}_{k,\ell}\subseteq \ol{H}_k^\ell
\]
such that the system $\{ \ol{S}_{i,\underline{t}} \}_{i\in \ol{G}_{k,\ell}}$ corresponding to $\ol{G}_{k,\ell}$ satisfies the SSC, and
\begin{equation} \label{sizeofG1}
\lvert \ol{G}_{k,\ell}\rvert \geq 3^{-\ol{s}_k} (1/m)^{-k\ell(\ol{s}_k-\varepsilon)} = 3^{-\ol{s}_k} (1/m)^{k\ell\varepsilon} \lvert \ol{H}_k \rvert^{\ell}
\end{equation}
by (\ref{newsimdim}).  Fix such an $\ell \geq \ell_0$ and consider the set
\[
G_{k,\ell} = \{ ((i_1,j_1),\ldots,(i_{k\ell},j_{k\ell})) \in D^{k\ell} :(i_1,\ldots,i_{k\ell}) \in \ol{G}_{k,\ell}\}
\]
and observe that, since $H_k$ had uniform vertical fibres,
\begin{equation}
\lvert G_{k,\ell} \rvert = \bigg( \frac{\lvert H_k \rvert}{\lvert \ol{H}_k \rvert}\bigg)^\ell \, \lvert \ol{G}_{k,\ell}\rvert
\geq \lvert H_k \rvert^\ell \,3^{-\ol{s}_k} (1/m)^{k\ell\varepsilon} \label{Gest}
\end{equation}
by (\ref{sizeofG1}).  Let $r = (1/n)^{k\ell}$ and consider the set
\[
F_0:= \ \bigcup_{\lambda \in G_{k,\ell}} S_{\lambda, \underline{t}}(F_{\underline{t}}) \ \subseteq F_{\underline{t}}.
\]
(Note that $F_0$ depends on $k, \ell$ and $\underline{t}$, but we do not display this dependence.) We will adopt the $\rho$-grid definition of $N_\rho( \cdot)$.   It follows immediately from the definition of box dimension that there exists a constant $C_\varepsilon >0$ depending only on $\varepsilon$ such that for all $\rho \in (0, 1]$ we have
\begin{equation} \label{simplebox}
  N_\rho\big(\pi(F_{\underline{t}})\big)\geq C_\varepsilon \, \rho^{-(\ol{s} - \varepsilon)}.
\end{equation}
Notice that each set $S_{\lambda, \underline{t}}(F)$ in the composition of $F_0$ is contained in the rectangle $S_{\lambda, \underline{t}}\big([0,1]^2\big)$ which has height $r$ and base length $(1/m)^{k\ell}$.  It follows that
\begin{equation} \label{projest}
N_r(S_{i,\underline{t}}(F_{\underline{t}})) \ \geq  \ N_{r (1/m)^{-k\ell}}\big(\pi(F_{\underline{t}})\big) \ \geq \  C_\varepsilon \Bigg(\frac{(1/m)^{k\ell}}{r}\Bigg)^{\ol{s}-\varepsilon}
\end{equation}
by (\ref{simplebox}).  Let $U$ be any closed square of sidelength $r$.  Since $\{S_{\lambda,\underline{t}}\big([0,1]^2\big)\}_{\lambda \in G_{k,\ell}}$ is a collection of rectangles which can only intersect at the boundaries  each with shortest side having length $r$, it is clear that $U$ can intersect no more than 9 of the sets $\{ S_{\lambda,\underline{t}}(F_{\underline{t}})\}_{\lambda \in G_{k,\ell}}$.  It follows that
\[
\sum_{\lambda \in G_{k,\ell}} N_{r} \big(S_{\lambda,\underline{t}}(F_{\underline{t}})\big) \  \leq  \ 9 \, N_{r} \Bigg(\bigcup_{\lambda\in G_{k,\ell}} S_{\lambda,\underline{t}}(F_{\underline{t}}) \Bigg)  \ \leq \  9 \, N_{r}(F_{\underline{t}}).
\]
This yields
\begin{align*}
N_{r}(F_{\underline{t}})& \geq  \tfrac{1}{9} \  \sum_{\lambda \in G_{k,\ell}} N_{r} \big(S_{i,\underline{t}}(F)\big) \\
& \geq  \tfrac{1}{9} \  \lvert G_{k,\ell} \rvert  C_\varepsilon \Bigg(\frac{(1/m)^{k\ell}}{r}\Bigg)^{\ol{s}-\varepsilon}  & \text{by (\ref{projest})}\\
& \geq  \tfrac{C_\varepsilon}{9} \,3^{-\ol{s}_k} r^{-(s_k-\varepsilon)} \  \lvert H_k \rvert^\ell \ \Big( (1/m)^{\ol{s}} (1/n)^{s_k-\ol{s}}\Big)^{k\ell} & \text{by (\ref{Gest})}\\
& \geq  \tfrac{C_\varepsilon}{27} \, r^{-(s_k-\varepsilon)}  \ \Bigg( \lvert H_k \rvert  \ (1/m)^{k \ol{s}} (1/n)^{k(s_k-\ol{s})}\Bigg)^{\ell} \\
& =  \tfrac{C_\varepsilon}{27} \, r^{-(s_k-\varepsilon)}
\end{align*}
by the definition of $s_k$.  This is valid for all $\ell \geq \ell_0$ and hence
\[
\liminf_{\ell \to \infty} \ \frac{\log N_{(1/n)^{k\ell}}(F_{\underline{t}})}{-\log (1/n)^{k\ell} } \  \geq \  s_k-\varepsilon \  \geq \  s- 2 \varepsilon.
\]
Fortunately, letting $r$ tend to zero through the sequence $(1/n)^{k\ell}$ as $\ell \to \infty$ is sufficient to give a lower bound on the \emph{lower} box dimension of $ F_{\underline{t}}$, see \cite[Section 3.1]{falconer} and so, since $\varepsilon$ can be made arbitrarily small, this yields $\underline{\dim}_\text{B}F_{\underline{t}} \geq s$ as required.
\end{proof}

\section{Generalisations, remarks, and future work}

\subsection{A generalisation}
\label{subsecgeneralisation}

We note that the fact we used reciprocals of integers $1/m$ and $1/n$ as the principle contractions in the defining system was not important.  We could equally well have chosen arbitrary $a,b \in (0,1/2]$ with $a>b$.  Moreover, we could allow different arrangements of the $a \times b$ rectangles in each fixed column provided they do not overlap, i.e. they need not be integer multiples of $b$ apart. See Figure \ref{fig:generalpattern} for an example of a pattern of this more general type. Thus, we have the following result.

\begin{figure}[H]
	\centering
	\includegraphics[width=80mm]{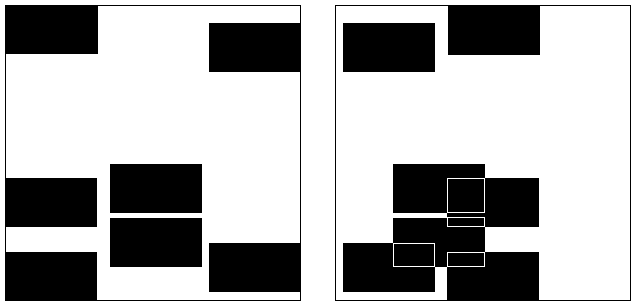}
\caption{A more general column pattern to which our results apply, and a concrete realization in which the columns overlap.} \label{fig:generalpattern}
\end{figure}

\begin{thm} \label{generalisation}
Let $0<b<a \leq 1/2$. Suppose there are numbers $\{ w_{ij} \}_{1\le i\le m, 1\le j\le N_i}$ for some integers $2\le m\le 1/a, N_i\ge 1$ such that $0\le w_{ij}\le 1-b$ and $|w_{i j_1}-w_{i j_2}|\ge b$ for all $i,j, j_1\neq j_2$.

Given $\underline{t}\in\mathbb{R}^m$, let $F_{\underline{t}}$ be the attractor of the IFS $\{ S_{(i,j),\underline{t}}\}_{1\le i\le m, 1\le j\le N_i}$, where
\[
S_{(i,j),\underline{t}}(x,y) = (ax ,by) + (t_i, w_{ij}).
\]
Then for all $\underline{t}\in [0,1-a]^m$ such that the IFS $\{ a x+t_i\}_{i=1}^m$ does not have super exponential concentration of cylinders, we have
\begin{align*}
\dim_{\text{\emph{H}}}(F_{\underline{t}}) &= \frac{\log \sum_{i=1}^m N_i^{\log a/ \log b}}{\log(1/a)}, \\
\dim_{\text{\emph{B}}}(F_{\underline{t}}) &= \frac{\log m}{\log(1/a)} + \frac{\log(\sum_{i=1}^m N_i/m)}{\log(1/b)}.
\end{align*}
In particular, this holds for all $\underline{t}$ outside of an exceptional set $E$ (depending only on $a$) of Hausdorff and packing dimension $m-1$.

Moreover, if $a$ is algebraic, then it also holds for all algebraic $\underline{t}$ such that the IFS $\{ a x+t_i\}_{i=1}^m$ does not have an exact overlap.
\end{thm}

We make some remarks on the assumptions of the above theorem. The restrictions $w_{ij}\in [0,1-b]$ and $\underline{t}\in [0,1-a]$ are not essential; they simply make sure that $F_{\ut}$ is a subset of the unit square, which can always be achieved by a change of coordinates.  The hypothesis $|w_{i j_1}-w_{i j_2}|\ge b$ guarantees that the rectangles in each column are non-overlapping, and this is an obvious necessary condition in general. The assumption $m\le 1/a$ is meant to ensure that the similarity dimension of the projected self-similar set (and measure) is at most $1$, and we require $a\le 1/2$ so that $m\le 1/a$ is not a vacuous assumption. For the box dimension calculation, these are not essential restrictions: if $m> 1/a$ and $a\in (0,1)$ is arbitrary, the proof goes through just by replacing $\log m/\log(1/a)$ by $1$ at the points where the dimension of the projection comes up (in the proofs of both the lower and upper bound), to give the same result with the formula for the box dimension replaced by
\[
1+\frac{\log(a\sum_{i=1}^m N_i)}{\log(1/b)}.
\]
For Hausdorff dimension, however, the result fails if $a>1/2$. Recall that a \emph{Pisot number} is an algebraic integer $>1$ such that all its algebraic conjugates are $<1$ in modulus. It is well known that Pisot numbers accumulate to $2$. Take $m=2$, $1/a$ to be any Pisot number in $(1,2)$, any $b\in (0,a)$ and $N_1=N_2=1$. Note that the translations do not play a role when $m=2$ (as long as the maps do not have the same fixed point, which is a co-dimension one phenomenon in parameter space). In this case it was shown by Przytycki and Urba\'{n}ski in \cite{prurb} that the Hausdorff dimension drops from the ``expected'' value, see also \cite[Theorem 15]{shmerkinoverlapping} for a simpler proof using McMullen's method from \cite{mcmullen}. In fact, the latter proof shows that the same phenomenon holds if $1/a\in (1,2)$ is Pisot, for any $m\ge 2$ and any $N_i$. We note that the issue here is that $a>1/2$; it would be interesting to understand the behaviour of Hausdorff dimension when $m>1/a$ but $a <1/2$.

\subsection{Final remarks}

There are various other directions in which this work could  move.  A  further generalisation in the direction of Theorem \ref{generalisation} would be to consider Lalley-Gatzouras type columns \cite{lalley-gatz}, which would allow for rectangles of varying heights and widths.  We do not see any difficulty in extending our arguments to cover this setting, but do not pursue it here to aid clarity of exposition. One could also consider random versions of the more general self-affine carpets considered by Bara\'nski \cite{baranski}, Feng-Wang \cite{fengaffine} or Fraser \cite{fraserboxlike}, however, in these cases our random model seems less natural as the dimension can depend on both principal projections, rather than just $\pi$.

In this article we have focused on the case with ``column alignment'' where the dimension drops from the affinity dimension, but Theorems \ref{mainH}, \ref{mainB} and \ref{generalisation} hold also when each column has just one rectangle, i.e. there are no special alignments, and in this case the dimension formulas we obtain coincide with the affinity dimension of the respective systems. Once one gives up the alignment, it makes sense to consider arbitrary self-affine systems, including those for which there is no dominant direction for all maps. We hope to address this situation in a forthcoming paper, leading to an improvement on Falconer's classical theorem from \cite{affine} in the case of diagonal maps.

Another interesting direction for further work would be to consider self-affine measures supported on our carpets.  Then one could ask if, for example, the Hausdorff dimension or $L^q$-spectrum was almost surely equal to the  Hausdorff dimension or $L^q$-spectrum when the columns do not overlap. For the Hausdorff dimension of self-affine measures, the proof of Theorem \ref{mainH} should apply with minor changes to yield an analogous result. On the other hand, $L^q$-spectra behave more like box counting dimension, and our methods clearly do not work here as we heavily relied on taking subsystems, which does not work for measures, but only for sets.

One could also look at different notions of dimension other than just the Hausdorff, packing and box dimensions considered here.  For example, the Assouad dimension $\dim_\text{A}$, and its natural dual the lower dimension $\dim_\text{L}$, have recently been gaining some attention in the literature on fractals and in particular overlapping self-similar sets \cite{fraseretal} and self-affine carpets \cite{mackay, fraserassouad}.  The definitions of these dimensions are quite technical and so we do not give them here, but rather refer the reader to the papers \cite{larman, fraserassouad}.  One of the key properties of our construction is that the box and Hausdorff dimensions can never be larger than the box and Hausdorff dimensions of the original Bedford-McMullen carpet. We conclude this section by briefly pointing out via two simple examples that this is not the case for Assouad and lower dimension.  This is based on the recent work of Mackay \cite{mackay} and Fraser \cite{fraserassouad} who computed these dimensions for certain classes of self-affine carpets.

\begin{thm}[Fraser, Mackay] \label{frasermackay}
Let $F$ be a standard Bedford-McMullen carpet.  Then
\[
\dim_\text{\emph{A}} F \ = \  \frac{\log \lvert \ol{D} \rvert}{\log m} \, + \, \max_{i=1, \dots, m} \frac{\log N_i}{\log n}
\]
and
\[
\dim_\text{\emph{L}} F \ = \  \frac{\log \lvert \ol{D} \rvert}{\log m} \, + \, \min_{i=1, \dots, m} \frac{\log N_i}{\log n}.
\]
\end{thm}

We will now use this theorem to provide examples showing that the Assouad and lower dimension can increase from the original values upon translation of columns.  The iterated function systems and their attractors will be given in the following figures.  In all cases we choose $m=3$ and $n=4$.

\begin{figure}[H]
	\centering
	\includegraphics[width=80mm]{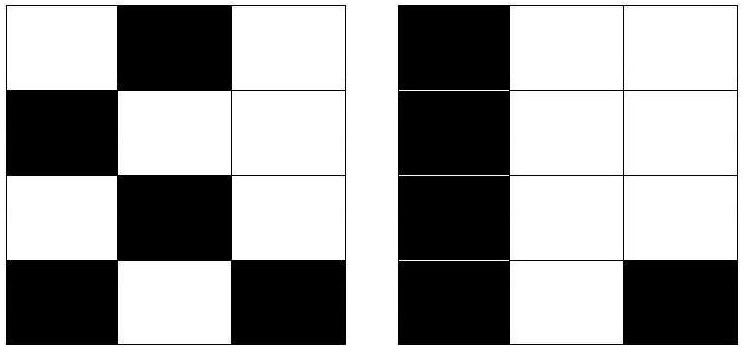}
\caption{Two Bedford-McMullen IFSs with attractors $F_1$ (left) and $F_2$ (right).  Note that we can translate the columns in the carpet on the left to obtain the carpet on the right.}
\end{figure}

By Theorem \ref{frasermackay}, we have
\[
\dim_\text{A} F_1 \ = \  1 \, + \, \frac{\log 2}{\log 4} \ < \ \frac{\log 2}{\log 3} \, + \, 1 \ = \ \dim_\text{A} F_2.
\]

\begin{figure}[H]
	\centering
	\includegraphics[width=80mm]{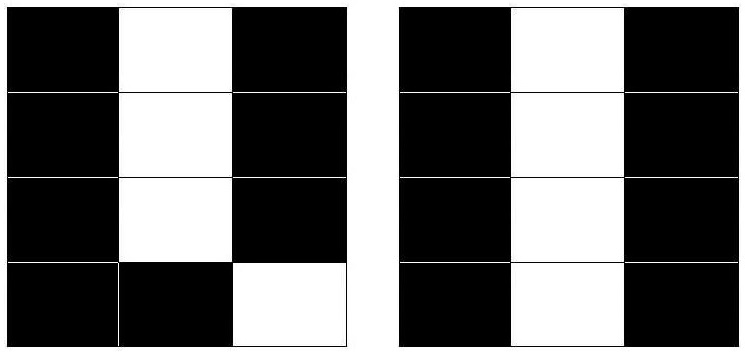}
\caption{Two Bedford-McMullen IFSs with attractors $F_3$ (left) and $F_4$ (right).  Note that we can translate the columns in the carpet on the left to obtain the carpet on the right.}
\end{figure}

By Theorem \ref{frasermackay}, we have
\[
\dim_\text{L} F_3 \ = \  1  \ < \ \frac{\log 2}{\log 3} \, + \, 1 \ = \ \dim_\text{L} F_4.
\]
Throughout this paper we relied on being able to understand the dimension of the projection onto the first coordinate, which is a self-similar subset of the unit interval, typically with overlaps.  The Assouad dimension and lower dimension also depend on this, however, the Assouad dimension of a self-similar subset of the unit interval with overlaps does not necessarily equal the Hausdorff dimension.  In \cite{fraseretal}, it was recently shown that  in the cases when the Assouad dimension is strictly greater than the Hausdorff dimension, then it is automatically equal to 1, no matter how small the Hausdorff dimension is.  The lower dimension, on the other hand, always coincides with the Hausdorff dimension, see \cite[Theorem 2.11]{fraserassouad}.

\section*{Acknowledgements}

This work began whilst P.S. was visiting J.M.F. at the University of St Andrews.  The work of J.M.F. was supported by the EPSRC grant EP/J013560/1 whilst at Warwick and an EPSRC doctoral training grant whilst at St Andrews. P.S. acknowledges support from Project PICT 2011-0436 (ANPCyT).


\end{document}